\documentclass[11pt,reqno]{amsart}
\usepackage[T1]{fontenc}
\usepackage{lmodern}
\usepackage{amsmath,amssymb,amsthm,mathtools}
\usepackage[margin=1.05in]{geometry}
\usepackage{microtype}
\usepackage[colorlinks=true,linkcolor=blue,citecolor=blue,urlcolor=blue]{hyperref}
\hypersetup{
  pdftitle={Prime approximation on irrational lines},
  pdfauthor={Zirui Jiang},
  pdfsubject={Metric Diophantine approximation on irrational lines},
  pdfkeywords={prime approximation, irrational lines, Selberg sieve}
}
\numberwithin{equation}{section}
\newtheorem{theorem}{Theorem}[section]
\newtheorem{proposition}[theorem]{Proposition}
\newtheorem{lemma}[theorem]{Lemma}

\theoremstyle{remark}

\newcommand{\PP}{\mathbb P}
\newcommand{\RR}{\mathbb R}
\newcommand{\ZZ}{\mathbb Z}
\newcommand{\NN}{\mathbb N}

\newcommand{\ind}{\mathbf 1}

\newcommand{\dist}[1]{\lVert #1\rVert}

\title{Metric Diophantine approximation with primes on irrational lines}
\author{Zirui Jiang}
\address{Department of Mathematics and Statistics, University of Nevada, Reno, USA}
\email{ziruij@unr.edu}

\begin{document}
\begin{abstract}
Let $c>1$ be irrational. We prove that, for almost every $x>0$
and every $\varepsilon>0$, there are infinitely many triples
$(p,r,s)$, with $p$ and $r$ prime and $s$ an integer, for which
\[
0<px-r<p^{-1/3+\varepsilon},\qquad
0<pcx-s<p^{-1/3+\varepsilon}.
\]
This improves the exponent $1/5$ in the homogeneous-line case
of a theorem of Baier and Ghosh. The arithmetic input is the
quantitative prime-counting lower bound in Matom\"aki's proof
concerning primes in shrinking neighborhoods of an irrational
rotation. We obtain a local first-moment lower bound by fixing
the prime numerator and applying the prime number theorem to
the admissible prime denominators. Two Selberg sieve squares,
together with explicit trigonometric interval approximations,
give a pointwise upper bound whose error tends to zero in
$L^1$ after normalization. A truncation and density-point
argument then yields the almost-everywhere conclusion.
\end{abstract}
\maketitle

\section{Introduction}

Let $\PP$ denote the set of positive rational primes. Metric Diophantine approximation with restrictions on both the numerator and the denominator has been studied by Harman \cite{Harman1988} and Jones \cite{Jones2001}. Imposing an additional relation between the coordinates leads to approximation on a lower-dimensional set. Harman and Jones \cite{HJ2002} treated certain nonlinear curves with three prime variables. For the line $(x,cx)$, Baier and Ghosh \cite[Theorem~1.2]{BG2015} proved that, for every irrational $c>1$ and every $\varepsilon>0$, almost every positive $x$ admits infinitely many solutions of
\begin{equation}\label{eq:bg}
0<px-r<p^{-1/5+\varepsilon},\qquad
0<pcx-s<p^{-1/5+\varepsilon},
\end{equation}
where $p,r\in\PP$ and $s\in\ZZ$. Their later work \cite{BG2017} considers restricted simultaneous approximation in higher dimensions. Inhomogeneous analogues are investigated in \cite{BG2018}.

We improve the exponent in \eqref{eq:bg} for homogeneous lines.

\begin{theorem}\label{thm:main}
Let $c>1$ be irrational. There is a Lebesgue-null set $E_c\subset(0,\infty)$ such that, for every $x\in(0,\infty)\setminus E_c$ and every $\varepsilon>0$, there exist infinitely many triples
\[
(p,r,s)\in\PP\times\PP\times\ZZ
\]
such that
\begin{equation}\label{eq:main}
0<px-r<p^{-1/3+\varepsilon},\qquad
0<pcx-s<p^{-1/3+\varepsilon}.
\end{equation}
\end{theorem}

The main arithmetic input is due to Matom\"aki
\cite{Matomaki2009}. For each fixed $0<\kappa<1/3$, her argument
provides a dyadic lower bound for primes $r\in(X,2X]$ satisfying
$\dist{cr}\ll X^{-\kappa}$, along scales determined by convergents
to $c$. To obtain a lower bound on a fixed subinterval
$I\subset[A,B]$, we first fix such a prime numerator $r$ and let
the prime denominator vary over an interval depending on $r$
and $I$. The prime number theorem, combined with the lengths
of the resulting approximation intervals, gives the local
first-moment estimate in Proposition~\ref{prop:lower}.
This lower bound is proportional to $|I|$, with a positive
constant independent of $I$, although the threshold for $X$
may depend on $I$. In particular, no asymptotic formula for
the good prime numerators in smaller intervals is required.

For the upper bound, we estimate the integrated error in
counting simultaneous fractional-part conditions with
divisibility restrictions. A normalized high power of a
geometric sum supplies interval majorants and minorants,
while the convergent approximation to $c$ controls the
nonzero Fourier frequencies within the chosen degree range.
Two explicit Selberg squares majorize the prime restrictions
on the denominator and the first numerator. With the notation
introduced below, Proposition~\ref{prop:upper} gives
\[
\frac{F_X(x)}{G_X}\le K+\frac{J_X(x)}{G_X},
\qquad
\left\|\frac{J_X}{G_X}\right\|_{L^1([A,B])}\longrightarrow0
\]
along the convergent scales. Truncating $F_X/G_X$ at $K$
therefore changes its local integrals by a quantity tending
to zero. The local lower bound and the density-point argument
in Lemma~\ref{lem:metric} then yield the almost-everywhere
conclusion. This is the same general metric strategy used
in \cite{HJ2002,BG2015}; the counting order and the arithmetic
scale are different.

Throughout, $e(t)=\exp(2\pi i t)$, $\{t\}=t-\lfloor t\rfloor$, and $\dist{t}$ is the distance from $t$ to the nearest integer. We write $[d,e]$ for the least common multiple. The symbols $\mu$ and $\varphi$ denote the M\"obius and Euler functions. All integrals and measures are with respect to Lebesgue measure.

\section{The arithmetic input and the counting scale}

We record the quantitative consequence of Matom\"aki's argument that will be used below. It is important here to use the prime-counting lower bound in her proof, rather than only the infinitude assertion of her theorem.

\begin{lemma}[Matom\"aki's prime-counting input]\label{lem:matomaki}
Let $c$ be irrational and $0<\kappa<1/3$ be fixed. If $a/Q$ is a convergent to $c$ and
\begin{equation}\label{eq:scale}
X=Q^{2/(1+\kappa)},
\end{equation}
then, as $Q$ tends to infinity through convergent denominators,
\begin{equation}\label{eq:primeinput}
\#\{r\in\PP:X<r\le2X,\ \dist{cr}\le5X^{-\kappa}\}
\gg_{c,\kappa}\frac{X^{1-\kappa}}{\log X}.
\end{equation}
\end{lemma}

\begin{proof}[Derivation from \cite{Matomaki2009}]
Put $R=X^{(1-\kappa)/2}$. The set introduced in \cite[\S1, (1.2)]{Matomaki2009} consists of integers $n\in(X,2X]$, coprime to $Q$, for which $an$ has a residue representative in $[1,3R]$. In \cite[\S2]{Matomaki2009}, this condition is assigned a smooth weight $h$, with $0\le h\le1$, supported on $[R/2,3R]$. The final lower bound in \cite[\S5, p.~283]{Matomaki2009} is
\[
\gg \frac{RX}{Q\log X}=\frac{X^{1-\kappa}}{\log X}
\]
for the resulting weighted prime count. For large $X$, the support of $h$ is contained in $[1,3R]$ and $3R<Q$. Each prime therefore has at most one admissible residue representative and weight at most one, so the same unweighted lower bound follows. For every counted prime $r$,
\[
\dist{cr}\le \left|c-\frac aQ\right|r+\dist{\frac{ar}{Q}}
\le\frac{2X}{Q^2}+\frac{3R}{Q}=5X^{-\kappa}.
\]
This proves \eqref{eq:primeinput}. The argument in the cited paper starts with an arbitrary sufficiently large convergent denominator, so the same scales are available here.
\end{proof}

Fix, for the remainder of the proof,
\begin{equation}\label{eq:fixed}
0<A<B,\qquad c>1\text{ irrational},\qquad 0<\kappa<1/3.
\end{equation}
Write $a_n/Q_n$ for the continued-fraction convergents to $c$, in lowest terms. We use
\[
Q_n\longrightarrow\infty,\qquad (a_n,Q_n)=1,\qquad
\left|c-\frac{a_n}{Q_n}\right|<\frac1{Q_n^2}.
\]
The variable $X$ runs through $X_n=Q_n^{2/(1+\kappa)}$. Thus $X_n\to\infty$ and
$Q_n=X_n^{(1+\kappa)/2}$. Removing finitely many initial convergents will not
change any of the limiting arguments. When working at one scale we write $a,Q,X$
in place of $a_n,Q_n,X_n$. Set
\begin{equation}\label{eq:etaG}
\eta=20X^{-\kappa},\qquad
N=\left\lceil\frac{4X}{A}\right\rceil,\qquad
G_X=\frac{X\eta^2}{(\log X)^2}.
\end{equation}
We always take $X$ large enough that $\eta<1/2$. For $x\in[A,B]$, let $F_X(x)$ count the prime denominators $p$ for which there exist $r\in\PP$ and $s\in\ZZ$ satisfying
\begin{equation}\label{eq:FX}
\frac{X}{2B}\le p\le\frac{4X}{A},\qquad
0<px-r<\eta,\qquad0<pcx-s<\eta.
\end{equation}
Since $\eta<1$, each $p$ admits at most one possible pair $(r,s)$, so $F_X$ also counts the corresponding triples. It is a finite sum of indicators of countable unions of open sets, and $0\le F_X\le N$; in particular it is measurable and integrable on every bounded interval. For $x\in[A,B]$, the restriction on $p$ forces both prime variables to tend to infinity with $X$.

\section{Local lower bounds by fixing the prime numerator}

In this section we derive a local integral lower bound from
Lemma~\ref{lem:matomaki}. We fix the prime numerator first and use the
prime number theorem to count admissible prime denominators.
Summing the lengths of the corresponding approximation intervals
yields Proposition~\ref{prop:lower}, with a positive constant
independent of the chosen subinterval.

\begin{proposition}\label{prop:lower}
Assume the prime-counting input \eqref{eq:primeinput} along the scales
\eqref{eq:scale}. There exists $c_0>0$, depending only on $A,B,c,\kappa$,
such that, for each fixed interval $I=[\alpha,\beta]\subset[A,B]$ with
$\alpha<\beta$,
\begin{equation}\label{eq:lower}
\int_I F_X(x)\,dx\ge c_0|I|G_X
\end{equation}
for all sufficiently large $X$ in \eqref{eq:scale}. The threshold for $X$
may depend on $I$; the constant $c_0$ does not.
\end{proposition}

\begin{proof}
Choose $M>0$, independent of $I$, so that the prime-counting input gives,
for all sufficiently large admissible $X$, a finite set
$\mathcal R_X\subset\PP\cap(X,2X]$ satisfying
\begin{equation}\label{eq:good-numerator-count}
\#\mathcal R_X\ge M\frac{X^{1-\kappa}}{\log X}.
\end{equation}
For each $r\in\mathcal R_X$, choose $s_r\in\ZZ$ with
\[
e_r=cr-s_r,\qquad |e_r|\le 5X^{-\kappa}=\eta/4.
\]

We first obtain a lower bound for the prime denominators using an interval
strictly inside the permitted range. Put
\begin{equation}\label{eq:prime-interior-endpoints}
a_I=\frac{3/\beta+1/\alpha}{4},\qquad
b_I=\frac{1/\beta+3/\alpha}{4}.
\end{equation}
Then
\[
0<\frac1\beta<a_I<b_I<\frac1\alpha.
\]
For $X<r\le2X$, define the finite set
\[
\mathcal P_r=\{p\in\PP:a_Ir<p\le b_Ir\}.
\]
If $\eta\le1$ and $X>1/(\beta a_I-1)$, then
$r(\beta a_I-1)>1\ge\eta$. Consequently every $p\in\mathcal P_r$
satisfies the strict bounds
\begin{equation}\label{eq:pinterval}
\frac{r+\eta}{\beta}<a_Ir<p\le b_Ir<\frac r\alpha.
\end{equation}
These bounds hold simultaneously for all $X<r\le2X$ once $X$ is
sufficiently large for the fixed interval $I$.

Write $\pi(y)=\#\{p\in\PP:p\le y\}$. For the fixed positive constants
$a_I<b_I$, the prime number theorem gives
\[
\frac{\log r}{r}\bigl(\pi(b_Ir)-\pi(a_Ir)\bigr)
\longrightarrow b_I-a_I\qquad(r\longrightarrow\infty).
\]
Thus, for all sufficiently large $r$,
\[
\#\mathcal P_r\ge\frac{b_I-a_I}{2}\frac r{\log r}.
\]
Since $p\le b_Ir$ for every $p\in\mathcal P_r$, it follows directly that
\[
\sum_{p\in\mathcal P_r}\frac1p
\ge\frac{\#\mathcal P_r}{b_Ir}
\ge\frac{b_I-a_I}{2b_I\log r}.
\]
For $X\ge2$ and $X<r\le2X$, we have
$\log r\le\log(2X)\le2\log X$. Moreover,
\[
\frac{b_I-a_I}{4b_I}
=\frac{\beta-\alpha}{2(\alpha+3\beta)}
\ge\frac{\beta-\alpha}{8B}.
\]
We therefore obtain, uniformly for $X<r\le2X$,
\begin{equation}\label{eq:reciprocalp}
\sum_{p\in\mathcal P_r}\frac1p
\ge\frac{b_I-a_I}{4b_I\log X}
\ge\frac{\beta-\alpha}{8B\log X}
\end{equation}
for all sufficiently large $X$, with a threshold that may depend on $I$.

We next turn each pair $r\in\mathcal R_X$, $p\in\mathcal P_r$ into a
set of good parameters. Define
\[
t_0=\frac{\eta}{2c},\qquad t_1=\frac{3\eta}{4c}.
\]
Since $c>1$, we have $0<t_0<t_1<\eta$. If $t_0<t<t_1$, then
\[
0<t<\eta,\qquad
0<\frac\eta4<ct+e_r<\eta.
\]
It follows that every $x$ in the open interval
\begin{equation}\label{eq:goodx}
J_{r,p}=\left(\frac{r+t_0}{p},\frac{r+t_1}{p}\right)
\end{equation}
satisfies
\[
0<px-r<\eta,\qquad 0<pcx-s_r<\eta.
\]
Indeed, take $t=px-r$ and use
$pcx-s_r=c(px-r)+e_r$. The denominator bounds \eqref{eq:pinterval}
place $J_{r,p}$ inside $(\alpha,\beta)$: its left endpoint is greater
than $r/p>\alpha$, and its right endpoint is less than
$(r+\eta)/p<\beta$. They also imply
\[
\frac{X}{2B}\le p\le\frac{4X}{A},
\]
since $p>(r+\eta)/\beta>X/B$ and $p<r/\alpha\le2X/A$.
Thus these are genuine solutions counted by $F_X$, and
\[
|J_{r,p}|=\frac{t_1-t_0}{p}=\frac{\eta}{4cp}.
\]

For completeness, the summation of these contributions is pointwise valid.
Suppose that $x$ belongs to two selected intervals with the same denominator
$p$. Their integer numerators $r,r'$ both satisfy $0<px-r,px-r'<\eta$,
so $|r-r'|<\eta\le1$ and hence $r=r'$. The corresponding integer
$s$ is likewise unique under the second approximation inequality.
Consequently distinct active pairs $(r,p)$ have distinct denominators
counted by $F_X$, and
\begin{equation}\label{eq:lower-indicator-domination}
\sum_{r\in\mathcal R_X}\sum_{p\in\mathcal P_r}\ind_{J_{r,p}}(x)
\le F_X(x).
\end{equation}
All these functions are measurable and integrable on $I$: the two sums are
finite, and $F_X$ is a measurable function bounded between $0$ and $N$.
Integrating \eqref{eq:lower-indicator-domination} term by term, and then
using \eqref{eq:good-numerator-count} and \eqref{eq:reciprocalp}, gives
\begin{align*}
\int_I F_X(x)\,dx
&\ge\frac{\eta}{4c}
\sum_{r\in\mathcal R_X}\sum_{p\in\mathcal P_r}\frac1p\\
&\ge\frac{\eta}{4c}\,
M\frac{X^{1-\kappa}}{\log X}\,
\frac{\beta-\alpha}{8B\log X}\\
&=\frac{M}{640cB}(\beta-\alpha)\,
\frac{X\eta^2}{(\log X)^2}.
\end{align*}
The last equality uses $X^{1-\kappa}=X\eta/20$. Hence we may take
$c_0=M/(640cB)>0$. This constant is independent of $I$; only the
thresholds used to fit the inner prime window and apply the prime number
theorem depend on $\alpha,\beta$.
\end{proof}

\section{Fourier estimates for rectangles}

This section develops the Fourier estimates needed for the sieve upper
bound. Lemma~\ref{lem:concentration-kernel} constructs a normalized
nonnegative trigonometric kernel and bounds its Fourier coefficients
and tail integrals. Lemma~\ref{lem:interval} uses this kernel to construct
pointwise majorants and minorants for interval indicators, with explicit
control of their constant coefficients.
Lemma~\ref{lem:kernel-parameters} selects a common kernel and verifies
the required width and degree restrictions uniformly over the sieve
moduli. Lemma~\ref{lem:geometric} bounds the integral of the absolute
value of a geometric exponential sum, while
Lemma~\ref{lem:frequencies} uses a rational approximation to $c$ to
separate the nonzero Fourier frequencies from zero and bound the sum
of their reciprocals. Finally, Lemma~\ref{lem:rectangle} combines the
polynomial and frequency estimates to bound the integrated error in
counting simultaneous fractional-part conditions with divisibility
restrictions, keeping the Fourier degree separate from the error
in the constant coefficients.

\begin{lemma}\label{lem:concentration-kernel}
Let $L,m\ge1$ be integers, and put
\[
S_L(x)=\sum_{j=0}^{L-1}e(jx),\qquad
Z_{L,m}=\int_0^1|S_L(x)|^{2m}\,dx,\qquad
K_{L,m}(x)=\frac{|S_L(x)|^{2m}}{Z_{L,m}}.
\]
Then $K_{L,m}$ is a nonnegative, even, real trigonometric polynomial,
of period one and degree at most
\[
D=m(L-1).
\]
Its integral over one period is one, and
\begin{equation}\label{eq:kernel-mass}
Z_{L,m}\ge\frac{1}{4\pi L}\left(\frac L2\right)^{2m}>0.
\end{equation}
For $0<\rho\le1/2$,
\begin{equation}\label{eq:kernel-tail}
\int_\rho^{1-\rho}K_{L,m}(x)\,dx
\le C_m(L\rho)^{-(2m-1)},\qquad C_m=8\pi\,2^{2m}.
\end{equation}
With the convention
$\widehat K(h)=\int_0^1K(x)e(-hx)\,dx$, its coefficients satisfy
\begin{equation}\label{eq:kernel-coefficients}
\widehat K_{L,m}(0)=1,\qquad
|\widehat K_{L,m}(h)|\le1\quad(h\in\ZZ),\qquad
\widehat K_{L,m}(h)=0\quad(|h|>D).
\end{equation}
\end{lemma}

\begin{proof}
For $x\ge0$, the inequality $|e(t)-1|\le2\pi|t|$ gives
\[
|S_L(x)-L|
\le\sum_{j=0}^{L-1}|e(jx)-1|
\le2\pi L^2x.
\]
Consequently $|S_L(x)|\ge L/2$ for
$0\le x\le(4\pi L)^{-1}$. Integrating over this interval, which is
contained in $[0,1]$, proves \eqref{eq:kernel-mass}.
Nonnegativity and normalization follow immediately. Moreover,
\[
|S_L(x)|^{2m}
=\left(\sum_{j=0}^{L-1}e(jx)\right)^m
 \left(\sum_{j=0}^{L-1}e(-jx)\right)^m.
\]
Every frequency in this finite expansion is the difference of two sums
of $m$ integers in $[0,L-1]$, and hence lies in $[-D,D]$.
This also proves the asserted trigonometric-polynomial representation.
The equality $S_L(-x)=\overline{S_L(x)}$ proves evenness.

For $0<x\le1/2$, the geometric-sum identity and
$\sin(\pi x)\ge2x$ give
\[
|S_L(x)|\le\frac{1}{2x}\le\frac1x.
\]
Since $|S_L(1-x)|=|S_L(x)|$, we obtain, writing $k=2m-1\ge1$,
\[
\begin{split}
\int_\rho^{1-\rho}|S_L(x)|^{2m}\,dx
&\le2\int_\rho^{1/2}x^{-2m}\,dx\\
&=\frac2k\bigl(\rho^{-k}-(1/2)^{-k}\bigr)
\le2\rho^{-k}.
\end{split}
\]
Division by the lower bound in \eqref{eq:kernel-mass} yields
\eqref{eq:kernel-tail}, with the displayed constant $C_m$.
Finally, orthogonality extracts the coefficients of the finite expansion,
while positivity and unit mass give
$|\widehat K_{L,m}(h)|\le\int_0^1K_{L,m}=1$.
\end{proof}

\begin{lemma}\label{lem:interval}
Let $L,m\ge1$, $0<\rho\le1/2$, and $0<\ell\le1$ satisfy
$2\rho\le\ell$ and
\begin{equation}\label{eq:interval-tail-condition}
\int_\rho^{1-\rho}K_{L,m}(y)\,dy\le\rho.
\end{equation}
Define the periodic open-interval indicator
\[
\chi_\ell(x)=\ind_{(0,\ell)}(\{x\}),
\]
and, using ordinary integrals over the indicated real intervals, set
\begin{equation}\label{eq:interval-polynomials}
\begin{split}
U_\ell(x)&=\int_{-\rho}^{\ell+\rho}K_{L,m}(x-y)\,dy+\rho,\\
V_\ell(x)&=\int_\rho^{\ell-\rho}K_{L,m}(x-y)\,dy-\rho.
\end{split}
\end{equation}
These are real trigonometric polynomials of degree at most $D=m(L-1)$.
For every real $x$, including all endpoints modulo one,
\begin{equation}\label{eq:interval-pointwise}
V_\ell(x)\le\chi_\ell(x)\le U_\ell(x),\qquad U_\ell(x)\ge0.
\end{equation}
Writing $t=3\rho$, their constant coefficients are
\begin{equation}\label{eq:interval-masses}
\widehat U_\ell(0)=\ell+t,\qquad
\widehat V_\ell(0)=\ell-t,
\end{equation}
and, for every integer $h$,
\begin{equation}\label{eq:interval-coefficients}
|\widehat U_\ell(h)|\le\ell+t,\qquad
|\widehat V_\ell(h)|\le\ell+t.
\end{equation}
No nonnegativity assertion is made for $V_\ell$.
\end{lemma}

\begin{proof}
Write $K=K_{L,m}$. For $a\le b$, put
\[
\mathcal C_{a,b}(x)=\int_a^bK(x-y)\,dy
=\int_{x-b}^{x-a}K(y)\,dy.
\]
It is nonnegative and one-periodic. If $b-a\le1$, its integral is over a
subinterval of a period, so $\mathcal C_{a,b}(x)\le1$.
Also, periodicity and unit mass give
\begin{equation}\label{eq:near-tail-mass}
\int_{-\rho}^{\rho}K(y)\,dy+
\int_\rho^{1-\rho}K(y)\,dy=1.
\end{equation}
We can reduce each pointwise assertion to $w=\{x\}\in[0,1)$.

If $0<w<\ell$, the interval
$[w-\ell-\rho,w+\rho]$ contains $[-\rho,\rho]$.
Equations \eqref{eq:interval-tail-condition} and
\eqref{eq:near-tail-mass} therefore imply
$U_\ell(w)\ge1$. Outside $(0,\ell)$ the upper bound follows from
$U_\ell\ge0$.

For the lower bound, if $0<w<\ell$, then
$\mathcal C_{\rho,\ell-\rho}(w)\le1$, since
$0\le\ell-2\rho\le1$, and consequently $V_\ell(w)\le1$.
If $w\in[\ell,1)$, the integration interval
\[
[w-\ell+\rho,w-\rho]\subset[\rho,1-\rho]
\]
has nonnegative length. Its $K$-integral is at most $\rho$, so
$V_\ell(w)\le0$.
The remaining possibility is $w=0$. In this case periodicity shifts the
integration interval to
\[
[-\ell+\rho,-\rho]+1
=[1-\ell+\rho,1-\rho]\subset[\rho,1-\rho],
\]
and gives the same bound. This proves \eqref{eq:interval-pointwise}
at both endpoints as well as at every other point. In particular, the
integrals in \eqref{eq:interval-polynomials} handle passage across zero
without an endpoint convention for a periodic characteristic function.

To verify the degree and the coefficient assertions, integrate the finite
Fourier expansion of $K$ term by term. For every integer $h$,
\[
\widehat{\mathcal C_{a,b}}(h)
=\widehat K(h)\int_a^b e(-hy)\,dy.
\]
Thus its degree is at most $D$, its constant coefficient is $b-a$, and
\[
|\widehat{\mathcal C_{a,b}}(h)|\le b-a.
\]
Adding or subtracting a real constant changes only the zero coefficient.
Formula \eqref{eq:interval-masses} follows, and the triangle inequality
gives
\[
|\widehat U_\ell(h)|\le\ell+2\rho+\rho=\ell+3\rho,
\qquad
|\widehat V_\ell(h)|\le\ell-2\rho+\rho
=\ell-\rho\le\ell+3\rho.
\]
This argument remains valid when the minorant, or its constant
coefficient, is negative.
\end{proof}

\begin{lemma}\label{lem:kernel-parameters}
Fix $0<\kappa<1/3$, and define
\begin{equation}\label{eq:kernel-parameters}
\delta=\frac{1-3\kappa}{200},\qquad z=X^\delta,\qquad
\eta=20X^{-\kappa},\qquad
\rho=\eta z^{-14},\qquad
L=\left\lceil\frac{z^{20}}\eta\right\rceil.
\end{equation}
Choose an integer $m\ge1$, depending only on $\kappa$, such that
\begin{equation}\label{eq:kernel-order}
6(2m-1)\delta>\kappa+14\delta,
\end{equation}
and put $D=m(L-1)$ and $t=3\rho$.
Then, for all sufficiently large $X$, the same kernel $K_{L,m}$ satisfies
Lemma~\ref{lem:interval} simultaneously for
$\ell=\eta/v$ and $\ell=\eta$, for every integer $1\le v\le z^2$.
Moreover,
\begin{equation}\label{eq:kernel-width-separation}
t\le\eta/v,\qquad
1\le D\ll_\kappa z^{20}/\eta,\qquad
vD<Q/2,\qquad D<Q,
\end{equation}
where $Q=X^{(1+\kappa)/2}$.
\end{lemma}

\begin{proof}
Because $\delta>0$, one can fix $m$ satisfying
\eqref{eq:kernel-order}; equivalently, it is enough to choose
$m>(\kappa+20\delta)/(12\delta)$.
The choice is independent of $X$, $u$ and $v$.
The radius $\rho=20X^{-\kappa-14\delta}$ tends to zero, and
\[
L\rho\ge\frac{z^{20}}\eta\,\eta z^{-14}=z^6.
\]
Writing $k=2m-1$, the ratio of the upper bound in
\eqref{eq:kernel-tail} to $\rho$ is therefore at most
\[
\frac{C_m z^{-6k}}{\rho}
=\frac{C_m}{20}\,
X^{\kappa+(14-6k)\delta}\longrightarrow0
\]
by \eqref{eq:kernel-order}. Hence
\eqref{eq:interval-tail-condition} holds for large $X$.

For $z\ge2$ and $1\le v\le z^2$ we have $4v\le z^{14}$, whence
\[
4\rho\le\frac\eta v\le\eta.
\]
In particular $2\rho<\eta/v$, $t\le\eta/v$, and, once $\eta<1/2$,
\[
\frac\eta v+2\rho\le\frac32\frac\eta v<1.
\]
Taking $v=1$ gives the corresponding assertions for $\ell=\eta$.
Thus both intervals satisfy every width and endpoint hypothesis in
Lemma~\ref{lem:interval}.

Finally, $z^{20}/\eta\to\infty$, so $L\ge2$ and
$L\le2z^{20}/\eta$ eventually. Since $m$ is fixed,
\[
\frac{z^2D}{Q}\le\frac{mz^2L}{Q}
\ll_\kappa X^{22\delta-(1-\kappa)/2}
=X^{(-78+34\kappa)/200}\longrightarrow0.
\]
This proves the degree bound and $vD<Q/2$ uniformly for $v\le z^2$.
It also gives $D<Q$ for all sufficiently large $X$.
\end{proof}

\begin{lemma}\label{lem:geometric}
Let $A\le B$, let $M\ge1$ be an integer, and let $\xi\in\RR\setminus\{0\}$.
Put
\[
K_{A,B}=(B-A+1)\left(1+\frac1{\log2}\right).
\]
Then
\begin{equation}\label{eq:geometric}
\int_A^B\left|\sum_{j=1}^{M}e(j\xi x)\right|\,dx
\le K_{A,B}\left(1+\frac1{|\xi|}\right)\log(2M).
\end{equation}
\end{lemma}

\begin{proof}
Write
\[
H_M(t)=\left|\sum_{j=0}^{M-1}e(jt)\right|.
\]
The sum starting at $j=1$ differs by the factor $e(t)$, so it has the
same absolute value. The function $H_M$ is continuous, nonnegative,
even, and periodic with period one. The geometric-sum identity gives
\[
H_M(t)\le M,\qquad H_M(t)|e(t)-1|\le2.
\]
For $0\le t\le1/2$ we have
$|e(t)-1|=2\sin(\pi t)\ge4t$. Splitting the integral at $1/(2M)$
therefore gives
\[
\int_0^{1/2}H_M(t)\,dt
\le M\frac1{2M}+\int_{1/(2M)}^{1/2}\frac{dt}{2t}
=\frac12+\frac12\log M.
\]
Evenness and periodicity imply
\begin{equation}\label{eq:geometric-period}
\int_0^1 H_M(t)\,dt\le1+\log M
\le\left(1+\frac1{\log2}\right)\log(2M).
\end{equation}

For any nonnegative continuous function $H$ of period one and any
$\alpha\le\beta$, take $q=\lceil\beta-\alpha\rceil$. Since
$[\alpha,\beta]\subseteq[\alpha,\alpha+q]$, periodicity gives
\[
\int_\alpha^\beta H(t)\,dt
\le q\int_0^1H(t)\,dt
\le(\beta-\alpha+1)\int_0^1H(t)\,dt.
\]
Changing variables $t=\xi x$, and reversing the endpoints when $\xi<0$,
we obtain
\[
\int_A^B H_M(\xi x)\,dx
\le\left(B-A+\frac1{|\xi|}\right)\int_0^1H_M(t)\,dt.
\]
Combine this with \eqref{eq:geometric-period} and
\[
B-A+\frac1{|\xi|}\le(B-A+1)\left(1+\frac1{|\xi|}\right)
\]
to prove \eqref{eq:geometric}.
\end{proof}

\begin{lemma}\label{lem:frequencies}
Let $c\in\RR$, $a\in\ZZ$, and let $Q,u,v,D$ be positive integers such that
\[
\gcd(a,Q)=1,\qquad
\left|c-\frac aQ\right|<\frac1{Q^2},\qquad 2vD\le Q.
\]
All frequencies $u(h/v+ck)$ with $|h|,|k|\le D$ and $(h,k)\ne(0,0)$
are nonzero, and
\begin{equation}\label{eq:freq}
\sum_{\substack{|h|,|k|\le D\\(h,k)\ne(0,0)}}
\left(1+\frac1{|u(h/v+ck)|}\right)
\le18\left(D^2+\frac vu DQ\right).
\end{equation}
\end{lemma}

\begin{proof}
First let $b,w\in\ZZ$ with $0<|w|\le Q/2$. The integer $aw-bQ$
is nonzero: otherwise coprimality would imply $Q\mid w$, contrary
to $0<|w|<Q$. Hence
\begin{equation}\label{eq:rational-separation}
|cw-b|
\ge\frac{|aw-bQ|}{Q}-\left|c-\frac aQ\right||w|
\ge\frac1Q-\frac{|w|}{Q^2}
\ge\frac1{2Q}.
\end{equation}
If $0<|k|\le D$, apply this with $w=vk$ and $b=-h$ to obtain
\begin{equation}\label{eq:shifted-separation}
|h+cvk|\ge\frac1{2Q},\qquad
|u(h/v+ck)|\ge\frac{u}{2Qv}.
\end{equation}
The frequencies with $k=0$ and $h\ne0$ are also nonzero.

Fix a nonzero $k$ with $|k|\le D$, and put $s=cvk$. Among the
integers $h$ with $|h+s|<1$ there are at most two. Indeed, writing
$s=\lfloor s\rfloor+\theta$ with $0\le\theta<1$ shows that such an
$h$ must belong to
$\{-\lfloor s\rfloor,-\lfloor s\rfloor-1\}$.
By \eqref{eq:shifted-separation}, each of these two terms contributes
at most $2Q$ to the reciprocal sum. Every other term contributes at
most one. Consequently
\[
\sum_{|h|\le D}\frac1{|h+cvk|}\le4Q+2D+1\le6Q,
\]
where we used $2D\le2vD\le Q$ and $Q\ge1$.
For $k=0$, excluding $h=0$, we instead have
\[
\sum_{0<|h|\le D}\frac1{|h|}\le2D\le6Q.
\]
Thus, using the exact rescaling
\[
\frac1{|u(h/v+ck)|}=\frac vu\,\frac1{|h+cvk|},
\]
and applying the corresponding bound to each of the $2D+1\le3D$
values of $k$, we obtain
\[
\sum_{\substack{|h|,|k|\le D\\(h,k)\ne(0,0)}}
\frac1{|u(h/v+ck)|}\le18\frac vu DQ.
\]
The number of summands is at most $(2D+1)^2\le9D^2$.
Adding their constant contributions proves \eqref{eq:freq}.
\end{proof}

For positive integers $N,u,v$, define
\begin{equation}\label{eq:Sdef}
S_{u,v}(x)=\#\left\{1\le n\le N:
\begin{array}{l}
u\mid n,\quad v\mid\lfloor nx\rfloor,\\
0<\{nx\}<\eta,\quad 0<\{cnx\}<\eta
\end{array}\right\}.
\end{equation}
Every positive integer is understood to divide zero. The function $S_{u,v}$
is measurable, since it is a finite sum of indicators of measurable sets,
and $0\le S_{u,v}\le N$. In particular, both $S_{u,v}$ and the remainder
defined below are integrable on every bounded interval.

\begin{lemma}\label{lem:rectangle}
Let $0<A<B$, let $N,u,v,D,Q$ be positive integers with $u\le N$, and let
$a\in\ZZ$ satisfy
\[
(a,Q)=1,\qquad \left|c-\frac aQ\right|<Q^{-2},
\qquad 2vD\le Q.
\]
Suppose that $0<\eta<1$ and $0\le t\le\eta/v$. Put
$\ell_1=\eta/v$ and $\ell_2=\eta$. For $i=1,2$, suppose that there are
real-valued trigonometric polynomials $U_i,V_i$ of degree at most $D$
such that, for every $y\in\RR$,
\[
V_i(y)\le\ind_{(0,\ell_i)}(\{y\})\le U_i(y),
\]
and such that their Fourier coefficients satisfy
\begin{align*}
\widehat U_i(0)&=\ell_i+t,&
\widehat V_i(0)&=\ell_i-t,\\
|\widehat U_i(h)|&\le\ell_i+t,&
|\widehat V_i(h)|&\le\ell_i+t\qquad (|h|\le D).
\end{align*}
The pointwise inequalities include the endpoints of the open intervals.
Define
\[
R_{u,v}(x)=S_{u,v}(x)-\frac{N\eta^2}{uv},
\qquad
K_{A,B}=(B-A+1)\left(1+\frac1{\log 2}\right).
\]
Then
\begin{equation}\label{eq:Rbound}
\begin{split}
\int_A^B|R_{u,v}(x)|\,dx
\le{}&5(B-A)N\eta t+(B-A)\\
&+216\eta^2K_{A,B}\log(2N)
\left(\frac{D^2}{v}+\frac{DQ}{u}\right).
\end{split}
\end{equation}
In particular, there is no additional requirement relating $t$ to $1/D$.
\end{lemma}

\begin{proof}
Write $M=\lfloor N/u\rfloor\ge1$. For a positive integer $v$ and any real
$y$, the assumptions $0<\eta<1$ give the exact equivalence
\[
\bigl(v\mid\lfloor y\rfloor\ \hbox{and}\ 0<\{y\}<\eta\bigr)
\quad\Longleftrightarrow\quad
0<\{y/v\}<\eta/v.
\]
Indeed, in the forward direction write $y=vm+r$ with $m\in\ZZ$ and
$0<r<\eta$. In the reverse direction put $m=\lfloor y/v\rfloor$ and
$r=v\{y/v\}$. Then $y=vm+r$ and $0<r<\eta<1$, so
$\lfloor y\rfloor=vm$ and $\{y\}=r$. Strict inequalities ensure that
the equivalence also treats all boundary cases correctly.
Reindexing the positive multiples of $u$ by $n=uj$, we therefore obtain
\[
S_{u,v}(x)=\sum_{j=1}^M
\chi_1(ujx/v)\chi_2(cujx),\qquad
\chi_i(y)=\ind_{(0,\ell_i)}(\{y\}).
\]

For two independent variables $y_1,y_2$, set
\begin{align*}
P^+(y_1,y_2)&=U_1(y_1)U_2(y_2),\\
P^-(y_1,y_2)&=V_1(y_1)U_2(y_2)
+U_1(y_1)V_2(y_2)-U_1(y_1)U_2(y_2).
\end{align*}
Then $P^-\le\chi_1\chi_2\le P^+$ pointwise. For the lower inequality,
the following identity, with arguments suppressed, has only nonnegative
terms on its right-hand side:
\[
\chi_1\chi_2-P^-=
(U_1-\chi_1)(U_2-\chi_2)
+(\chi_1-V_1)U_2+U_1(\chi_2-V_2).
\]
Here $U_i\ge\chi_i\ge0$. The two-dimensional constant coefficients are
\begin{align*}
\widehat P^+(0,0)&=\frac{\eta^2}{v}
+t\left(\frac\eta v+\eta\right)+t^2,\\
\widehat P^-(0,0)&=\frac{\eta^2}{v}
-t\left(\frac\eta v+\eta\right)-3t^2.
\end{align*}
Thus each differs from $\eta^2/v$ by at most
\[
\Delta=t\left(\frac\eta v+\eta\right)+3t^2\le5\eta t,
\]
because $t\le\eta/v\le\eta$. Moreover, every coefficient of either
$P^+$ or $P^-$ has absolute value at most
\[
3\left(\frac\eta v+t\right)(\eta+t)
\le\frac{12\eta^2}{v}.
\]

Evaluate these finite Fourier expansions at
$(y_1,y_2)=(ujx/v,cujx)$ and sum over $1\le j\le M$. The frequencies
that result from a pair $(h,k)$ are
\[
\xi_{h,k}=u(h/v+ck),\qquad |h|,|k|\le D.
\]
The constant term corresponds to $(h,k)=(0,0)$. Taking absolute values
of the remaining finite sums gives
\[
\left|S_{u,v}(x)-\frac{M\eta^2}{v}\right|
\le5M\eta t+
\frac{12\eta^2}{v}
\sum_{\substack{|h|,|k|\le D\\(h,k)\ne(0,0)}}
\left|\sum_{j=1}^M e(j\xi_{h,k}x)\right|.
\]
Since $0\le N/u-M<1$ and $\eta^2/v\le1$, it follows for every
$x\in[A,B]$ that
\begin{equation}\label{eq:Rpoint}
|R_{u,v}(x)|\le5N\eta t+1+
\frac{12\eta^2}{v}
\sum_{\substack{|h|,|k|\le D\\(h,k)\ne(0,0)}}
\left|\sum_{j=1}^M e(j\xi_{h,k}x)\right|.
\end{equation}

By Lemma~\ref{lem:frequencies}, all the frequencies in
\eqref{eq:Rpoint} are nonzero and
\[
\sum_{\substack{|h|,|k|\le D\\(h,k)\ne(0,0)}}
\left(1+\frac1{|\xi_{h,k}|}\right)
\le18\left(D^2+\frac vu DQ\right).
\]
Lemma~\ref{lem:geometric} bounds the integral of each corresponding
geometric sum by
\[
K_{A,B}\left(1+\frac1{|\xi_{h,k}|}\right)\log(2M).
\]

All terms in \eqref{eq:Rpoint} are integrable on $[A,B]$. Integrate that
inequality, apply the preceding two bounds, and use
$\log(2M)\le\log(2N)$.
The product of the coefficient constant $12$ and the frequency constant
$18$ is $216$, giving \eqref{eq:Rbound}.
\end{proof}

\section{Two Selberg squares and an integrated upper bound}

All divisor indices in this section are positive integers. For a real cutoff
$z\ge2$, write $Z=\lfloor z\rfloor$ and
\[
\mathcal G(z)=\sum_{d=1}^{Z}\frac{\mu(d)^2}{\varphi(d)}.
\]
In particular, $\mathcal G(z)\ge1$. The following finite construction supplies
both the sieve weights and the logarithmic lower bound needed below.

\begin{lemma}\label{lem:sieve}
For $z\ge2$ there are real numbers $\lambda_d$, indexed by integers $d\ge1$,
with $\lambda_d=0$ for $d>Z$, such that
\begin{equation}\label{eq:weightprops}
\lambda_1=1,\qquad |\lambda_d|\le d\quad(d\ge1),\qquad
\sum_{d=1}^{Z}\sum_{e=1}^{Z}
\frac{\lambda_d\lambda_e}{[d,e]}=\frac1{\mathcal G(z)}.
\end{equation}
Moreover, the explicit bound
\begin{equation}\label{eq:G-elementary}
\mathcal G(z)\ge\frac12\log(Z+1)\ge\frac12\log z
\end{equation}
holds.
\end{lemma}

\begin{proof}
For $1\le n\le Z$, put
\[
y_n=\frac{\mu(n)}{\varphi(n)\mathcal G(z)},
\qquad
\lambda_d=d\sum_{\substack{1\le n\le Z\\d\mid n}}
\mu(n/d)y_n\quad(d\ge1).
\]
The sum defining $\lambda_d$ is empty if $d>Z$. Since
$|\mu(n)|=\mu(n)^2$ and $|\mu(n/d)|\le1$, we have
\[
|\lambda_d|
\le\frac{d}{\mathcal G(z)}
\sum_{\substack{1\le n\le Z\\d\mid n}}
\frac{|\mu(n/d)|\,|\mu(n)|}{\varphi(n)}
\le\frac{d}{\mathcal G(z)}
\sum_{n=1}^{Z}\frac{\mu(n)^2}{\varphi(n)}=d.
\]
At $d=1$ the same definition gives
$\lambda_1=\mathcal G(z)/\mathcal G(z)=1$.

For every $1\le r\le Z$, finite M\"obius inversion over multiples gives
\begin{align*}
\sum_{\substack{1\le d\le Z\\r\mid d}}\frac{\lambda_d}{d}
&=\sum_{\substack{1\le n\le Z\\r\mid n}}y_n
  \sum_{\substack{d\mid n\\r\mid d}}\mu(n/d)\\
&=\sum_{\substack{1\le n\le Z\\r\mid n}}y_n
  \sum_{h\mid n/r}\mu(h)=y_r.
\end{align*}
Here the inner divisor sum is one when $n=r$ and zero otherwise. Using
$de=[d,e](d,e)$ and $\sum_{r\mid k}\varphi(r)=k$, we consequently obtain
\begin{align*}
\sum_{d=1}^{Z}\sum_{e=1}^{Z}\frac{\lambda_d\lambda_e}{[d,e]}
&=\sum_{d=1}^{Z}\sum_{e=1}^{Z}\frac{\lambda_d\lambda_e}{de}
  \sum_{r\mid(d,e)}\varphi(r)\\
&=\sum_{r=1}^{Z}\varphi(r)
  \left(\sum_{\substack{1\le d\le Z\\r\mid d}}
  \frac{\lambda_d}{d}\right)^2
=\sum_{r=1}^{Z}\varphi(r)y_r^2
=\frac1{\mathcal G(z)}.
\end{align*}

For the logarithmic bound, let $H_Z=\sum_{n=1}^{Z}1/n$. Every positive
integer can be written $n=b^2a$ with $a$ squarefree and $a,b\ge1$.
If $n\le Z$, then both $a$ and $b$ are at most $Z$. Dropping the restriction
$b^2a\le Z$ from this decomposition, and then using $\varphi(a)\le a$, gives
\begin{align*}
H_Z
&\le\left(\sum_{b=1}^{Z}\frac1{b^2}\right)
  \left(\sum_{\substack{1\le a\le Z\\a\ \mathrm{squarefree}}}\frac1a\right)\\
&\le2\sum_{\substack{1\le a\le Z\\a\ \mathrm{squarefree}}}
  \frac1{\varphi(a)}
=2\mathcal G(z).
\end{align*}
The bound on the first factor is elementary: for $b\ge2$,
$b^{-2}\le(b(b-1))^{-1}$, and the latter terms telescope to give
$\sum_{b=1}^{Z}b^{-2}\le2$. Finally,
$\log(Z+1)\le H_Z$ by integration of $1/t$, while
$z<Z+1$. This proves \eqref{eq:G-elementary} without a squarefree-density
asymptotic.
\end{proof}

\begin{proposition}\label{prop:upper}
There is a constant $K>0$ such that, for all sufficiently large convergent
scales $X$ in \eqref{eq:scale}, there is a nonnegative measurable function
$J_X$, integrable on $[A,B]$, for which
\begin{equation}\label{eq:upper}
F_X(x)\le KG_X+J_X(x)\quad(x\in[A,B]),
\qquad
\int_A^B J_X(x)\,dx=o(G_X).
\end{equation}
One may take $K=32/(A\delta^2)$, where
$\delta=(1-3\kappa)/200$.
\end{proposition}

\begin{proof}
Keep the parameters in \eqref{eq:fixed}--\eqref{eq:etaG}, and set
\begin{equation}\label{eq:parameters}
\delta=\frac{1-3\kappa}{200},\quad z=X^\delta,\quad
Z=\lfloor z\rfloor,\quad
L=\left\lceil\frac{z^{20}}\eta\right\rceil,\quad
\rho=\frac{\eta}{z^{14}},\quad t=3\rho,\quad D=m(L-1).
\end{equation}
Choose the fixed integer $m\ge1$ as in \eqref{eq:kernel-order}.
Lemma~\ref{lem:kernel-parameters} then supplies a common kernel and the
interval polynomials of Lemma~\ref{lem:interval}, uniformly for
$1\le v\le Z^2\le z^2$. Their widths are $\eta/v$ and $\eta$, their
constant-coefficient errors are $t=3\rho$, and their degree is at most
$D=m(L-1)$. Thus $L$ is the base kernel bandwidth and $D$ is the degree
used in every frequency-box estimate below. For all sufficiently large
$X$, the lemma gives
\[
0<t\le\frac{\eta}{v}\le\eta<\frac12,\qquad
1\le D\le\frac{2m z^{20}}\eta,\qquad
2vD\le Q,\qquad D<Q
\quad(1\le v\le Z^2).
\]
The explicit degree bound follows from $L\le2z^{20}/\eta$ in its proof.
In particular the relevant separation exponent is
\begin{equation}\label{eq:separationexp}
\frac{z^2D}{Q}
\le\frac{2m z^{22}}{\eta Q}
=\frac m{10}X^{22\delta-(1-\kappa)/2}\longrightarrow0,
\qquad
22\delta-\frac{1-\kappa}{2}=\frac{-78+34\kappa}{200}<0.
\end{equation}
Also $z^2/X=X^{2\delta-1}\to0$, so
$Z^2\le N$. In particular $\lfloor N/u\rfloor\ge1$ whenever
$1\le u\le Z^2$. Finally, we may assume
\begin{equation}\label{eq:sieve-counting-numerics}
N\le\frac{8X}{A},\qquad \log X\ge1,\qquad
\log(2N)\le2\log X.
\end{equation}
The first inequality follows from $\lceil y\rceil\le2y$ for $y\ge1$.
For the last inequality, take $X\ge16/A$ and use
$2N\le16X/A\le X^2$. These conditions depend only on the fixed parameters,
and hold simultaneously for every sufficiently large convergent scale.

Use the explicit weights in Lemma~\ref{lem:sieve} at cutoff $z$, and define
\[
W(k)=\left(\sum_{\substack{1\le d\le Z\\d\mid k}}\lambda_d\right)^2
\qquad(k\in\ZZ,\ k\ge0).
\]
This definition also covers $k=0$, since every positive integer divides
zero and the sum is finite. We have $W(k)\ge0$ for every such $k$.
If $k>z$ is prime, the only divisor present in the sum is $1$, and
$W(k)=1$.

The choice $0<\delta<1$ implies $z=o(X)$. Hence, for sufficiently large
$X$, the primes in every triple counted by $F_X(x)$ satisfy
\[
p\ge\frac{X}{2B}>z,\qquad
r>px-\eta\ge\frac{AX}{2B}-\eta>z
\qquad(x\in[A,B]).
\]
Moreover, $0<px-r<\eta<1$ implies $r=\lfloor px\rfloor$, and the
same inequalities determine $s=\lfloor pcx\rfloor$. Thus a fixed
denominator contributes at most one triple. For $1\le n\le N$ and
$x\in[A,B]$ we have $nx>0$, so $\lfloor nx\rfloor$ is a nonnegative
integer. The two prime squares therefore give the pointwise majorization
\begin{equation}\label{eq:majorant}
F_X(x)\le\sum_{n=1}^{N}
\ind_{(0,\eta)}(\{nx\})\ind_{(0,\eta)}(\{cnx\})
W(n)W(\lfloor nx\rfloor).
\end{equation}
Each counted denominator contributes exactly one to the right-hand side;
all other terms are nonnegative.

Expanding the two squares in \eqref{eq:majorant} gives exactly
\[
\sum_{d_1=1}^{Z}\sum_{d_2=1}^{Z}
\sum_{e_1=1}^{Z}\sum_{e_2=1}^{Z}
\lambda_{d_1}\lambda_{d_2}\lambda_{e_1}\lambda_{e_2}
S_{[d_1,d_2],[e_1,e_2]}(x),
\]
where $S_{u,v}$ is the actual finite count in \eqref{eq:Sdef}. Both
$u=[d_1,d_2]$ and $v=[e_1,e_2]$ are positive integers at most $Z^2$.
Define, without making any further choice,
\[
R_{u,v}(x)=S_{u,v}(x)-\frac{N\eta^2}{uv}.
\]
The total main term factors into two copies of the quadratic identity
\eqref{eq:weightprops}. By \eqref{eq:G-elementary} and
\eqref{eq:sieve-counting-numerics}, it satisfies
\begin{equation}\label{eq:sievemain}
N\eta^2\left(\sum_{d=1}^{Z}\sum_{e=1}^{Z}
\frac{\lambda_d\lambda_e}{[d,e]}\right)^2
=\frac{N\eta^2}{\mathcal G(z)^2}
\le\frac{32}{A\delta^2}\frac{X\eta^2}{(\log X)^2}
=KG_X.
\end{equation}
In particular the constant $K=32/(A\delta^2)$ is strictly positive.

Write $K_{A,B}=(B-A+1)(1+1/\log2)$. The general mass-error form of
Lemma~\ref{lem:rectangle}, with degree $D$ and mass error $t$, now gives
\begin{align}
\int_A^B|R_{u,v}(x)|\,dx
&\le5(B-A)N\eta t+(B-A)\notag\\
&\quad+216\eta^2K_{A,B}\log(2N)
\left(\frac{D^2}{v}+\frac{DQ}{u}\right)\notag\\
&\le5(B-A)N\eta t+(B-A)
+432\eta^2K_{A,B}\log(2N)DQ.
\label{eq:uniformR}
\end{align}
The final inequality uses $u,v\ge1$ and $D\le Q$. No condition
$t\le1/D$ is required: the constant-coefficient error is bounded using
$t\le\eta/v$, as verified above.

Define the concrete nonnegative remainder
\begin{equation}\label{eq:Jdefinition}
J_X(x)=\sum_{d_1=1}^{Z}\sum_{d_2=1}^{Z}
\sum_{e_1=1}^{Z}\sum_{e_2=1}^{Z}
|\lambda_{d_1}\lambda_{d_2}\lambda_{e_1}\lambda_{e_2}|
\,|R_{[d_1,d_2],[e_1,e_2]}(x)|.
\end{equation}
The maps $x\mapsto\lfloor nx\rfloor$ and $x\mapsto\{nx\}$ are Borel
measurable. Each $S_{u,v}$ is therefore a measurable finite sum of
indicators, bounded between $0$ and $N$. It follows that every
$R_{u,v}$ is measurable and bounded on $[A,B]$, and hence integrable there.
The finite sum \eqref{eq:Jdefinition} is likewise measurable, integrable,
and nonnegative. Taking absolute values only in the remainder terms of
the exact expansion proves $F_X(x)\le KG_X+J_X(x)$ pointwise.

Since $|\lambda_d|\le d\le Z$ on the summation range,
\[
\sum_{d_1=1}^{Z}\sum_{d_2=1}^{Z}
\sum_{e_1=1}^{Z}\sum_{e_2=1}^{Z}
|\lambda_{d_1}\lambda_{d_2}\lambda_{e_1}\lambda_{e_2}|
=\left(\sum_{d=1}^{Z}|\lambda_d|\right)^4\le Z^8\le z^8.
\]
Integrating \eqref{eq:Jdefinition} term by term and using
\eqref{eq:uniformR}, $t=3\eta/z^{14}$, $D\le2m z^{20}/\eta$, and
\eqref{eq:sieve-counting-numerics}, we obtain
\begin{align}
\int_A^B J_X(x)\,dx
&\le\frac{120(B-A)}{A}\frac{X\eta^2}{z^6}
+(B-A)z^8
+1728mK_{A,B}\eta Qz^{28}(\log X)^2\notag\\
&\le C_{A,B,m}\left(
\frac{X\eta^2}{z^6}+z^8+\eta Qz^{28}(\log X)^2\right),
\label{eq:Jbound}
\end{align}
where one possible positive constant is
$C_{A,B,m}=120(B-A)/A+(B-A)+1728mK_{A,B}$.
For the last term we used $\log X\le(\log X)^2$ after applying
$\log(2N)\le2\log X$; the resulting slightly weaker bound is convenient
for a single normalization.

Let
\[
E_X=\frac{X\eta^2}{z^6}+z^8+\eta Qz^{28}(\log X)^2.
\]
Using $\eta=20X^{-\kappa}$ and $Q=X^{(1+\kappa)/2}$ gives the exact ratio
\begin{align*}
\frac{E_X}{G_X}
&=X^{-6\delta}(\log X)^2
+\frac1{400}X^{8\delta-1+2\kappa}(\log X)^2\\
&\quad+\frac1{20}X^{28\delta+(3\kappa-1)/2}(\log X)^4.
\end{align*}
The first exponent is negative because $\delta>0$, and
\[
8\delta-1+2\kappa=\frac{47\kappa-24}{25}< -\frac13.
\]
For the last exponent,
\begin{equation}\label{eq:criticalexp}
28\delta+\frac{3\kappa-1}{2}
=-\frac9{25}(1-3\kappa)<0.
\end{equation}
Every term in the ratio thus tends to zero. Together with
\eqref{eq:sievemain} and \eqref{eq:Jbound}, this proves
\eqref{eq:upper}.
\end{proof}

\section{Completion of the metric argument}

We first state the measure-theoretic step with its integrability
requirements explicit. Upper bounds are allowed to hold only after
discarding finitely many indices. The threshold in a local lower bound
may depend on the interval.

\begin{lemma}\label{lem:metric}
Let $I_0=[A,B]$, where $A<B$. Let $f_j:I_0\to[0,\infty)$ be measurable,
let $g_j>0$, and let $J_j:I_0\to[0,\infty)$ be measurable. Suppose that
$K\ge0$ and $c_0>0$ are constants such that, for all sufficiently large
$j$, $J_j$ is Lebesgue integrable on $I_0$ and
\[
f_j\le Kg_j+J_j\quad\hbox{almost everywhere on }I_0,
\qquad
\frac1{g_j}\int_{I_0}J_j(x)\,dx\longrightarrow0.
\]
Suppose also that, for every fixed nondegenerate closed subinterval
$I\subset I_0$,
\[
\liminf_{j\to\infty}\frac1{g_j}\int_I f_j(x)\,dx\ge c_0|I|.
\]
If $E\subset I_0$ is measurable and, for every $x\in E$, one has
$f_j(x)=0$ for all sufficiently large $j$, then $|E|=0$.
\end{lemma}

\begin{proof}
Discard a finite initial segment so that the integrability and upper-bound
hypotheses hold for every remaining index, and reindex this tail by
$j\in\NN$. This does not change the limiting hypotheses or eventual
vanishing on $E$. The upper bound, nonnegativity, and integrability of
$J_j$ show that $f_j$ is also integrable on $I_0$: its absolute value
is dominated almost everywhere by the integrable function $Kg_j+J_j$.
Thus all integrals below are ordinary finite Lebesgue integrals.

Set
\[
h_j(x)=\min\{f_j(x)/g_j,K\}.
\]
These functions are measurable and satisfy $0\le h_j\le K$ everywhere.
Almost everywhere on $I_0$,
\[
0\le\frac{f_j}{g_j}-h_j\le\frac{J_j}{g_j}.
\]
Consequently, for every fixed subinterval $I$,
\[
0\le\frac1{g_j}\int_I f_j-\int_I h_j
\le\frac1{g_j}\int_{I_0}J_j\longrightarrow0,
\]
and hence
\[
\liminf_{j\to\infty}\int_Ih_j\ge c_0|I|.
\]
Equivalently, for every $\epsilon>0$, the inequality
$\int_I h_j\ge c_0|I|-\epsilon$ holds eventually, with a threshold
that may depend on $I$ and $\epsilon$.

For $x\in E$, eventual vanishing of $f_j(x)$ implies $h_j(x)\to0$.
The constant function $K$ is integrable on the finite-measure set
$I\cap E$, so dominated convergence gives
\[
\int_{I\cap E}h_j\longrightarrow0.
\]
Since $\int_{I\setminus E}h_j\le K|I\setminus E|$, we obtain
\begin{equation}\label{eq:density-complement}
c_0|I|\le K|I\setminus E|
\end{equation}
for every nondegenerate closed subinterval $I\subset I_0$.
If $|E|>0$, the Lebesgue density theorem supplies a density point
$x\in E\cap(A,B)$, since the two endpoints have measure zero. For
$I_r=[x-r,x+r]\subset I_0$, division of
\eqref{eq:density-complement} by $2r$ and passage to $r\downarrow0$
would give $c_0\le0$, a contradiction.
\end{proof}

\begin{proof}[Proof of Theorem~\ref{thm:main}]
Fix an irrational $c>1$ and invoke the prime-counting input stated in
Lemma~\ref{lem:matomaki}. We first prove the conclusion for one fixed
$\varepsilon$ with $0<\varepsilon<1/3$. Set
\[
\nu=\frac13-\varepsilon,\qquad
\nu<\kappa<\frac13,
\]
and fix $0<A<B$. Enumerate the convergents to $c$ as $a_j/Q_j$ and
take the actual scales
\[
X_j=Q_j^{2/(1+\kappa)}\longrightarrow\infty.
\]
All statements about sufficiently large scales in the preceding
propositions refer to this sequence.

When $\eta<1$, each prime denominator $p$ determines at most one
admissible pair $(r,s)$ in \eqref{eq:FX}. Thus the measurable function
$F_X$ counts denominators and triples with the same multiplicity. Its
finite-sum representation uses $1\le p\le\lceil4X/A\rceil$, retaining
$p\le4X/A$ as an exact condition in each summand. In particular,
$0\le F_X\le N$ and $F_X$ is integrable on $[A,B]$.

Proposition~\ref{prop:upper} supplies nonnegative integrable functions
$J_{X_j}$ and a constant $K$ such that eventually
\[
F_{X_j}\le K G_{X_j}+J_{X_j}
\quad\hbox{almost everywhere on }[A,B],\qquad
\frac1{G_{X_j}}\int_A^B J_{X_j}\longrightarrow0.
\]
In the construction, each $J_X$ is a finite sum of absolute values of
the measurable, integrable remainders $R_{u,v}$, multiplied by
nonnegative constant weight factors. Thus $J_X\in L^1([A,B])$.

Choose a single index $j_0$ so that for every $j\ge j_0$ one has
$X_j>1$, $\eta_j<1$, and all the upper-estimate and error-estimate
conditions hold. Define
\[
f_j=F_{X_{j+j_0}},\qquad
g_j=G_{X_{j+j_0}}>0,\qquad
J_j=J_{X_{j+j_0}}.
\]
These sequences satisfy the upper and integrability conditions of Lemma~\ref{lem:metric}
at every index. Proposition~\ref{prop:lower} gives, for every fixed
nondegenerate closed interval $I\subset[A,B]$,
\[
\int_I f_j\ge c_0|I|g_j
\quad\hbox{for all sufficiently large }j.
\]
The constant $c_0>0$ is independent of $I$. No common starting index
for all subintervals is required: the finite shift preserves each of
their separate eventual lower bounds.

Uniformly for the allowed denominators, $p\le(4/A)X$ and $\nu>0$
give
\[
p^{-\nu}\ge(4/A)^{-\nu}X^{-\nu},\qquad
\frac{20X^{-\kappa}}{(4/A)^{-\nu}X^{-\nu}}
=20(4/A)^\nu X^{\nu-\kappa}\longrightarrow0.
\]
Hence every triple counted by $F_X$ satisfies \eqref{eq:main} once
$X$ is sufficiently large.

Let $E\subset[A,B]$ consist of points with only finitely many triples
satisfying \eqref{eq:main} for this $\varepsilon$. This set is measurable.
Indeed, the triple set $\PP\times\PP\times\ZZ$ is countable; enumerate
its distinct elements as $t_1,t_2,\ldots$, and let $V_k$ be the open set
on which the strict inequalities for $t_k$ hold. The points with
infinitely many solutions form the Borel set
\[
\bigcap_{m=1}^\infty\ \bigcup_{k\ge m}V_k.
\]
For any fixed $x\in E$, the finitely many target triples have bounded
denominators. Since each denominator counted by $F_X(x)$ satisfies
$p\ge X/(2B)\to\infty$, the preceding comparison forces $F_X(x)=0$
for all sufficiently large scales. Lemma~\ref{lem:metric} now gives
$|E|=0$.

The intervals $[1/m,m]$, $m\ge2$, cover $(0,\infty)$, so a countable
union of the preceding exceptional sets proves the assertion on the
positive half-line for this fixed $\varepsilon\in(0,1/3)$. For an
arbitrary $\varepsilon>0$, choose
$0<\varepsilon'<\min\{\varepsilon,1/3\}$. Since $p\ge2$,
\[
p^{-1/3+\varepsilon'}\le p^{-1/3+\varepsilon},
\]
and every solution at $\varepsilon'$ is a solution at $\varepsilon$.
Thus the conclusion holds almost everywhere for each fixed positive
$\varepsilon$.

Finally, apply this conclusion to the countable sequence
$\varepsilon_n=1/(n+1)$, $n\ge0$. Let $E_c$ be the union of the
corresponding exceptional null sets. If $x>0$ and $x\notin E_c$, then
for any $\varepsilon>0$ we may choose $n$ with
$\varepsilon_n\le\varepsilon$ and use the same monotonicity to obtain
infinitely many solutions of \eqref{eq:main}. Therefore, for each fixed
irrational $c>1$, a single exceptional null set works simultaneously
for every $\varepsilon>0$, as asserted. The exceptional set is allowed
to depend on $c$.
\end{proof}

\section*{Acknowledgements}
I am deeply grateful to my advisor, Dr. Jing-Jing Huang,
for introducing me to the wonderland of the metric Diophantine
approximation and for his continued guidance and encouragement.

\end{document}